\pdfoutput=1
\documentclass[11pt]{amsart}

\usepackage{epsfig}
\usepackage{stmaryrd}
\usepackage{graphicx}
\usepackage{amscd}
\usepackage{amsmath}
\usepackage{amsfonts}
\usepackage{mathrsfs}
\usepackage{mathabx}
\usepackage{amssymb}
\usepackage{verbatim}
\usepackage{comment}
\usepackage{color}

\newtheorem{theorem}{Theorem}[section]
\theoremstyle{plain}

\newtheorem{corollary}[theorem]{Corollary}

\newtheorem{defi}[theorem]{Definition}

\newtheorem{lemma}[theorem]{Lemma}

\newtheorem{prop}[theorem]{Proposition}
\newtheorem{remark}[theorem]{Remark}

\numberwithin{equation}{section}

\def\nuhat{\widehat{\nu}}
\def\muhat{\widehat{\mu}}

\def\dist{{\rm dist}}

\def\Hk{{\mathcal H}}

\def\Ok{{\mathcal O}}

\def\half{\frac{1}{2}}

\newcommand{\lam}{\lambda}
\def\Lam{\Lambda}
\newcommand{\gam}{\gamma}

\def\b0{{\bf 0}}

\newcommand{\R}{{\mathbb R}}
\newcommand{\Q}{{\mathbb Q}}
\newcommand{\Z}{{\mathbb Z}}
\newcommand{\C}{{\mathbb C}}

\def\N{{\mathbb N}}

\def\bp{{\bf p}}
\def\wh{\widehat}

\def\Lk{{\mathcal L}}

\def\Dk{{\mathcal D}}

\def\beq{\begin{equation}}
\def\eeq{\end{equation}}
\def\wt{\widetilde}
\newcommand{\Ek}{{\mathcal E}}

\newcommand{\eps}{{\varepsilon}}

\def\ve1{\vec{1}}

\def\pb{\boldsymbol{p}}

\def\thb{\boldsymbol{\theta}}

\def\Spec{{\rm Spec}}
\def\Diag{{\rm Diag}}
\def\Dsc{{\mathscr D}}

\begin{document}

\title
{Fourier decay for homogeneous self-affine measures}

\author[B.\ Solomyak]{BORIS SOLOMYAK}

\address{Boris Solomyak, Department of Mathematics, Bar-Ilan University, Ramat Gan, 5290002 Israel}
\email{bsolom3@gmail.com}

\thanks{Supported in part by the Israel Science Foundation grant 911/19.}

\date{\today}

\begin{abstract}
We show that for  Lebesgue almost all $d$-tuples $(\theta_1,\ldots,\theta_d)$, with $|\theta_j|>1$,  any self-affine measure for a homogeneous non-degenerate iterated function system
$\{Ax+a_j\}_{j=1}^m$ in $\R^d$, where $A^{-1}$ is a diagonal matrix with the entries $(\theta_1,\ldots,\theta_d)$, has power Fourier decay at infinity.
\end{abstract}

\maketitle

\thispagestyle{empty}

\section{Introduction}

For a finite positive Borel measure $\mu$ on $\R^d$, consider the Fourier transform
$$
\wh\mu(\xi) = \int_{\R^d} e^{-2\pi i \langle \xi, x\rangle}\,d\mu(x).
$$
We are interested in the decay properties of $\muhat$ at infinity. The measure $\mu$ is called {\em Rajchman} if 
$$
\lim \muhat(\xi) = 0,\ \ \mbox{as}\ |\xi|\to \infty,
$$
where $|\xi|$ is a norm (say, the Euclidean norm) of $\xi\in \R^d$. 
Whereas absolutely continuous measures are Rajchman by the Riemann-Lebesgue Lemma, it is a subtle question to decide which singular measures are such, see, e.g., the survey of Lyons
\cite{Lyons}.
A much stronger property, useful for many applications is the following.

\begin{defi}{\em
For $\alpha>0$ let
$$
\Dsc_d(\alpha) = \bigl\{\nu\ \mbox{finite positive measure on $\R^d$}:\ \left|\widehat{\nu}(t)\right| = O_\nu(|t|^{-\alpha}),\ \ |t|\to \infty\bigr\},
$$
and denote $\Dsc_d= \bigcup_{\alpha>0} \Dsc_d(\alpha)$. A measure $\nu$ is said to have {\em power Fourier decay} if $\nu\in \Dsc_d$.}
\end{defi}

Many recent papers have been devoted to the question of Fourier decay for classes of ``fractal'' measures, see e.g., \cite{BD,JS,Li,LS1,LS2,SaSte,Sol_Fourier,Bremont,AHW20,VY20,Rap21}. Here we continue this line of research, focusing on the class of
{\em homogeneous self-affine measures} in $\R^d$.
A measure $\mu$ is called self-affine  if it is the invariant measure for a self-affine iterated function system (IFS) $\{f_j\}_{j=1}^m$, with $m\ge 2$, where $f_j(x) = A_jx + a_j$, the matrices $A_j:\R^d\to \R^d$ are invertible linear contractions (in some norm) and $a_j\in \R^d$ are   ``digit'' vectors.
This means that for some probability vector $\pb = (p_j)_{j\le m}$ holds
\beq \label{ssm1}
\mu = \sum_{j=1}^m p_j (\mu\circ f_j^{-1}).
\eeq
It is well-known that this equation defines a unique probability Borel measure. The self-affine IFS is {\em homogeneous} if all $A_j$ are equal to each other: $A  = A_j$ for $j\le m$.
Denote the digit set by $\Dk:=\{a_1,\ldots,a_m\}$ and the corresponding self-affine measure by $\mu(A,\Dk,\pb)$. We will write $\pb>0$ if all $p_j>0$.
Following \cite{Hochman}, we say that the IFS is {\em affinely irreducible} if the attractor is not contained in a proper affine subspace of $\R^d$. It is easy to see that this is a necessary condition
for the self-affine measure to be Rajchman, so this will always be our assumption. By a conjugation with a translation, we can always assume that $0\in \Dk$. In this case 
affine irreducibility is equivalent to the digit set $\Dk$  being a {\em cyclic family} for $A$, that is, $\R^d$ being the smallest $A$-invariant subspace containing $\Dk$.

The IFS is {\em self-similar} if all $A_j$ are contracting similitudes, that is, $A_j = \lam_j \Ok_j$ for some $\lam_j\in (0,1)$ and orthogonal matrices $\Ok_j$. In many aspects, ``genuine'' (i.e., non-self-similar) self-affine and self-similar IFS are very different; of course, the distinction exists only for $d\ge 2$.

Every homogeneous self-affine measure can be expressed as an infinite convolution product
\beq \label{eq-conv}
\mu(A,\Dk,\pb) = \Bigl(\Asterisk\prod\limits_{n=0}^\infty\Bigr) \sum_{j=1}^m p_j \delta_{A^n a_j},
\eeq
and for every $\pb>0$ it is supported on the attractor (self-affine set) 
$$
K_{A,\Dk}:= \Bigl\{x\in \R^d: \ x = \sum_{n=0}^\infty A^n b_n,\ b_n \in \Dk\Bigr\}.
$$
By the definition of the self-affine measure,
$$
\muhat(\xi) = \sum_{j=1}^m p_j \int e^{-2\pi i \langle \xi, Ax + a_j\rangle}\, d\mu = \Bigl(\sum_{j=1}^m p_j  e^{ -2\pi i \langle \xi,  a_j\rangle}\Bigr) \muhat(A^t\xi),
$$
where $A^t$ is the matrix transpose of $A$. Iterating we obtain
\begin{equation} \label{eq1}
\muhat(\xi) = \prod_{n=0}^\infty\left(\sum_{j=1}^m p_j  e^{-2\pi i \langle (A^t)^n\xi,  a_j\rangle}\right) = \prod_{n=0}^\infty\left(\sum_{j=1}^m p_j  e^{-2\pi i \langle \xi,  A^n a_j\rangle}\right),
\end{equation}
where the infinite product converges, since $\|A^n\| \to 0$ exponentially fast.


\subsection{Background}
We start with the known results on Fourier decay for classical Bernoulli convolutions $\nu_\lam$, namely, self-similar measures on the line, corresponding to the IFS $\{\lam x, \lam x+1\}$, with $\lam\in (0,1)$ and probabilities $(\half,\half)$ (often the digits
$\pm 1$ are used instead; it is easy to see that taking any two distinct digits results in the same measure, up to an affine change of variable).  Erd\H{o}s \cite{Erd1} proved that $\nuhat_\lam(t)\not\to 0$ as $t\to \infty$ when $\theta=1/\lam$ is a {\em Pisot number}.
Recall that a Pisot number is an algebraic integer greater than one, whose algebraic (Galois) conjugates are all less than one in modulus. Salem \cite{Salem} showed that if
$1/\lam$ is not a Pisot number, then  $\nuhat_\lam$ is a Rajchman measure.
In the other direction, Erd\H{o}s \cite{Erd2}  proved that for any $[a,b]\subset (0,1)$ there exists $\alpha>0$ such that
$\nu_\lam\in \Dsc_1(\alpha)$ for a.e.\ $\lam\in [a,b]$. Later, Kahane \cite{kahane} indicated that Erd\H{o}s' argument actually gives that $\nu_\lam\in \Dsc_1$ for all $\lam\in (0,1)$ outside a set of zero Hausdorff dimension. (We should mention that   very few specific $\lam$ are known, for which $\nu_\lam$ has power Fourier decay, see
Dai, Feng, and Wang \cite{DFW}.) In the original papers of Erd\H{o}s and Kahane there were no explicit quantitative bounds; this was done in the survey \cite{PSS00}, where the expression
``Erd\H{o}s-Kahane argument'' was used first.
The general case of a homogeneous self-similar measure on the line is treated analogously to Bernoulli convolutions: the self-similar measure is still an infinite convolution and the Erd\H{o}s-Kahane argument on power Fourier decay goes through with minor modifications, see \cite{DFW,ShSol16}.
Although one of the main motivations for the study of the Fourier transform has been  the question of absolute continuity/singularity of $\nu_\lam$, here we do not discuss it but refer the reader to the recent survey \cite{Var}.

Next we turn to the non-homogeneous case on the line. Li and Sahlsten \cite{LS1} proved that if $\mu$ is a self-similar measure on the line with contraction ratios
$\{r_i\}_{i=1}^m$ and there exist $i\ne j$ such that $\log r_i/\log r_j$ is irrational, then $\mu$ is Rajchman. Moreover, they showed logarithmic decay of the Fourier transform under a Diophantine condition. A related result for self-conformal measures was recently obtained by Algom, Rodriguez Hertz, and Wang \cite{AHW20}. Br\'emont \cite{Bremont} obtained an (almost) complete characterization of (non)-Rajchman self-similar measures in the case when $r_j = \lam^{n_j}$ for $j\le m$. To be non-Rajchman, it is necessary for $1/\lam$ to be Pisot. For ``generic'' choices of the probability vector $\bp$, assuming that $\Dk\subset \Q(\lam)$ after an affine conjugation, this is also sufficient, but there are some  exceptional cases of positive co-dimension. Varj\'u and Yu \cite{VY20} proved logarithmic decay of the Fourier transform in the case when $r_j = \lam^{n_j}$ for $j\le m$ and $1/\lam$ is algebraic, but not a Pisot or Salem number. In \cite{Sol_Fourier} we showed that outside a zero Hausdorff dimension exceptional set of parameters, all self-similar measures on $\R$ belong to $\Dsc_1$; however, the exceptional set is not explicit. 

Turning to higher dimensions, we mention the recent paper by Rapaport \cite{Rap21}, where he gives an
algebraic characterization of self-similar IFS for which there exists a probability vector yielding a non-Rajchman self-similar measure. Li and Sahlsten \cite{LS2} investigated 
self-affine measures in $\R^d$ and obtained power Fourier decay under some algebraic conditions, which never hold for a homogeneous self-affine IFS. Their main assumptions
are total irreducibility
of the closed group generated by the contraction linear maps $A_j$ and non-compactness of the projection of this group to  $PGL(d,\R)$. For $d=2,3$ they showed that this is sufficient.


\subsection{Statement of results} 
We assume that $A$ is a matrix diagonalizable over $\R$. Then we can reduce the IFS, via a linear change of variable, to one where $A$ is a diagonal matrix with real entries.
Given $A = \Diag[\theta_1^{-1},\ldots,\theta_d^{-1}]$, with $|\theta_j|>1$, a set of digits $\Dk = \{a_1,\ldots,a_m\}\subset \R^d$, and a probability vector $\pb$, we write 
$\thb= (\theta_1,\ldots,\theta_d)$ and
denote by $\mu(\thb,\Dk,\pb)$ the self-affine measure defined by \eqref{ssm1}. Our main motivation is the class of measures  which can be viewed as ``self-affine Bernoulli convolutions'', with $A = \Diag[\theta^{-1}_1,\ldots,\theta^{-1}_d]$ a diagonal matrix with distinct real entries and $\Dk = \{0, (1,\ldots,1)\}$. In this special case we denote the self-affine measure by $\mu(\thb,\pb)$.

\begin{theorem} \label{th-decay1}
There exists an exceptional set $E\subset \R^d$, with $\Lk^d(E)=0$, such that for all $\thb\in \R^d\setminus E$, with $\min_j|\theta_j|>1$, for all  sets of digits $\Dk$, such that the IFS is
affinely irreducible, and all  $\pb>0$, holds $\mu(\thb,\Dk,\pb)\in \Dsc_d$.
\end{theorem}

The theorem is a consequence of a more quantitative statement.

\begin{theorem} \label{th-decay2}
Fix $1 < b_1 < b_2 < \infty$ and $c_1,\eps >0$. Then there exist $\alpha>0$ and $\Ek\subset \R^d$, depending on these parameters, such that $\Lk^d(\Ek)=0$ and for all $\thb\not\in \Ek$ satisfying
$$
b_1 \le \min_j|\theta_j| < \max_j|\theta_j| \le b_2\ \ \ \mbox{and}\ \ \ |\theta_i - \theta_j| \ge c_1,\ i\ne j,
$$
for all digit sets $\Dk$ such that the IFS is affinely irreducible, and all $\pb$ such that $\min_j p_j \ge \eps$, we have $\mu(\thb,\Dk, \pb) \in \Dsc_d(\alpha)$.
\end{theorem}

\begin{proof}[Reduction of Theorem~\ref{th-decay1} to Theorem~\ref{th-decay2}.] 
For $M\in \N$ let $\Ek^{(M)}$ be the exceptional set obtained from Theorem~\ref{th-decay2} with $b_1 = 1 + M^{-1}, b_2 = M$, and $\eps = c_1 = M^{-1}$. Then the set
$$
E = \bigcup_{M=2}^\infty \Ek^{(M)} \cup \bigl\{\thb:\ \exists\,i\ne j, \ \theta_i = \theta_j\bigr\}.
$$
has the desired properties. \end{proof}

The proof of Theorem~\ref{th-decay2} uses a version of the Erd\H{o}s-Kahane technique. We follow the general scheme of \cite{PSS00,ShSol16}, but this is not a trivial extension.

In view of the convolution structure, Theorem~\ref{th-decay2} yields some information on absolute continuity of self-affine measures, by a standard argument.

\begin{corollary}
Fix $1 < b_1 < b_2 < \infty$ and $c_1,\eps >0$. Then there exist a sequence $n_k\to \infty$ and $\wt \Ek_k\subset \R^d$, depending on these parameters, such that $\Lk^d(\wt\Ek_k)=0$ 
and for all $\thb\not\in \wt\Ek_k$ satisfying
$$
b_1 \le \min_j|\theta^{n_k}_j| < \max_j|\theta^{n_k}_j| \le b_2\ \ \ \mbox{and}\ \ \ |\theta^{n_k}_i - \theta^{n_k}_j| \ge c_1,\ i\ne j,
$$
for all digit sets $\Dk$ such that the IFS is affinely irreducible, and all $\pb$ such that $\min_j p_j \ge \eps$, the measure $\mu(\thb,\Dk, \pb)$ 
is absolutely continuous with respect to $\Lk^d$, with a Radon-Nikodym derivative in $C^k(\R^d),\ k\ge 0$.
\end{corollary}

\begin{proof}[Proof (derivation)]
Let $n\ge 2$. It follows from \eqref{eq-conv} that
$$
\mu(A,\Dk,\bp) = \mu(A^n,\Dk,\bp) * \mu(A^n,A\Dk,\bp)\ldots * \mu(A^n, A^{n-1}\Dk,\bp).
$$
It is easy to see that if the original IFS is affinely irreducible, then so are the IFS associated with $(A^n,A^j\Dk)$, and moreover, these IFS are all affine conjugate to each other.
Therefore, if $\mu(A^n,\Dk,\bp)\in \Dsc_d(\alpha)$, then $\mu(A,\Dk,\bp) \in \Dsc_d(n\alpha)$. As is well-known,
$$
\mu\in \Dsc_d(\beta),\ \beta> d+k \implies \frac{d\mu}{d\Lk^d} \in C^k(\R^d),
$$
so we can take $n_k$ such that $n_k \alpha > d+k$, and $\wt \Ek_k = \{\thb: \thb^{n_k}\in \Ek\}$, where $\alpha$ and $\Ek$ are from Theorem~\ref{th-decay2}.
\end{proof}

\begin{remark} {\em 
{\bf (a)} In general, the power decay cannot hold for all $\thb$; for instance,  it is easy to see that the measure $\mu(\thb,\pb)$ is not Rajchman if at least one of $\theta_k$ is a Pisot number.
Thus in the most basic case with two digits, the  exceptional set has Hausdorff dimension  at least $d-1$.

{\bf (b)} It is natural to ask what happens if $A$ is not diagonalizable over $\R$. A complex eigenvalue of $A$ corresponds to a 2-dimensional homogeneous self-similar IFS with rotation, or an IFS of the form $\{\lam z + a_j\}_{j=1}^m$, with $\lam\in \C$, $|\lam|<1$, and $a_j \in \C$. In \cite{ShSol16b} it was shown that for all $\lam$ outside a set of Hausdorff dimension zero, the corresponding self-similar measure belongs to $\Dsc_2$. It may be possible to combine the methods of \cite{ShSol16b} with those of the current paper to obtain power Fourier decay for a typical $A$ diagonalizable over $\C$. It would also be interesting to consider the case of non-diagonalizable $A$, starting with a single Jordan block.

{\bf (c)} In the special case of $d=2$ and $m=2$, our system reduces to a planar self-affine IFS, conjugate to $\{(\lam x, \gam y) \pm (-1,1)\}$ for $0 < \gam < \lam < 1$. This system has been studied by many authors, especially the dimension and topological properties of its attractor, see \cite{HareSid17} and the references therein.
For our work, the most relevant is the paper by Shmerkin \cite{Shm06}. Among other results, he proved absolute continuity with a density in $L^2$ of the self-affine measure (with some fixed probabilities) almost everywhere in some region, in particular, in some explicit neighborhood of $(1,1)$.
He also showed that if $(\lam^{-1},\gam^{-1})$ for a {\em Pisot pair}, then the measure is not Rajchman and hence singular.
}
\end{remark}

\subsection{Rajchman self-affine measures}
The question ``when is $\mu(A,\Dk,\pb)$ is Rajchman?'' is not addressed here. Recently  Rapaport \cite{Rap21} obtained an (almost) complete characterization of {\em self-similar} Rajchman measures in $\R^d$. Of course, our situation is vastly simplified by the assumption that the IFS is homogeneous, but still it is not completely straightforward. The key notion here is the following.

\begin{defi}
A collection of numbers $(\theta_1,\ldots,\theta_m)$ (real or complex) is called a {\em Pisot family} or a {\em P.V.\ $m$-tuple} if 

{\bf (i)} $|\theta_j|>1$ for all $j\le m$ and

{\bf (ii)} there is a monic integer polynomial $P(t)$,  such that $P(\theta_j) = 0$ for all $j\le m$, whereas every other root $\theta'$ of $P(t)$ satisfies $|\theta'|<1$.
\end{defi}

It is not difficult to show, using the classical techniques of Pisot \cite{Pisot} and Salem \cite{Salem}, as well as some ideas from \cite[Section 5]{Rap21} that
\begin{itemize}
\item If $\mu(A,\Dk,\pb)$ is not a Rajchman measure and the IFS is affinely irreducible, then the spectrum $\Spec(A^{-1})$ contains a Pisot family;
\item if $\Spec(A^{-1})$ contains a Pisot family, then for a ``generic'' choice of $\Dk$, with $m\ge 3$, the measure $\mu(A,\Dk,\pb)$ is Rajchman; however,
\item if $\Spec(A^{-1})$ contains a Pisot family, then under appropriate conditions the measure $\mu(A,\Dk,\pb)$ is not Rajchman. For instance, this holds if there is at least one conjugate of the elements of the Pisot family less than 1 in absolute value, $m=2$, and $A$ is diagonalizable over $\R$.
\end{itemize}
We omit the details.


\section{Proofs}

The following is an elementary inequality.

\begin{lemma} Let $\bp = (p_1,\ldots,p_m)>0$ be a probability vector and $\alpha_1=0,\ \alpha_j \in \R$, $j=2,\ldots,m$. Denote $\eps = \min_j p_j$ and write $\|x\|=\dist(x,\Z)$. Then
for any $k\le m$,
\beq \label{elem}
\left|\sum_{j=1}^m p_j e^{-2\pi i \alpha_j}\right| \le 1 - 2\pi\eps \|\alpha_k\|^2.
\eeq
\end{lemma}

\begin{proof}
Fix $k\in \{2,\ldots,m\}$. We can estimate
$$
\left|\sum_{j=1}^m p_j  e^{-2\pi \alpha_j}\right|  = \left|p_1 + \sum_{j=2}^m p_j  e^{-2\pi \alpha_j}\right|   \le \left|p_1 + p_k e^{-2\pi i \alpha_k}\right| + (1-p_1 - p_k).
$$
Assume that $p_1\ge p_k$, otherwise, write $|p_1 + p_k e^{-2\pi i \alpha_k}| = |p_1 e^{2\pi i \alpha_k} + p_k|$ and repeat the argument. Then
observe that $|p_1 + p_k e^{-2\pi i \alpha_k}|\le (p_1-p_k) + p_k|1 + e^{-2\pi i \alpha_k}|$ and $|1 + e^{-2\pi i \alpha_k}| = 2|\cos(\pi \alpha_k)|\le 2(1 - \pi \|\alpha_k\|^2)$. This implies the desired inequality.
\end{proof}

Recall \eqref{eq1}:
$$
\muhat(\xi) = \prod_{n=0}^\infty\left(\sum_{j=1}^m p_j  e^{-2\pi i \langle \xi,  A^n a_j\rangle}\right).
$$
For $\xi\in \R^d$, with $\|\xi\|_\infty \ge 1$, let $\eta(\xi) = (A^t)^{N(\xi)} \xi$, where $N(\xi)\ge 0$ is maximal, such that $\|\eta(\xi)\|_\infty \ge 1$. Then $\|\eta(\xi)\|_\infty \in [1, \|A^t\|_\infty]$ and \eqref{eq1} implies
\begin{equation} \label{eq2}
\muhat(\xi) = \muhat(\eta(\xi))\cdot \prod_{n=1}^{N(\xi)} \left(\sum_{j=1}^m p_j  e^{-2\pi i \langle \eta(\xi),\,  A^{-n} a_j\rangle}\right).
\end{equation}

\subsection{Proof of Theorem~\ref{th-decay2}}
First we show that the case of a general digit set may be reduced to $\Dk = \{0, (1,\ldots,1)\}$.
We start with the formula \eqref{eq2}, which under the current assumptions becomes
$$
\muhat(\xi) = \muhat(\eta(\xi))\cdot \prod_{n=1}^{N(\xi)} \Bigl(\sum_{j=1}^m p_j  \exp\Bigl[-2\pi i \sum_{k=1}^d \eta_k a_j^{(k)} \theta_k^n\Bigr]\Bigr),
$$
where $a_j = (a_j^{(k)})_{k=1}^d$ and $\eta(\xi) = (\eta_k)_{k=1}^d$. Note that 
$\|\eta(\xi)\|_\infty \in [1, \max_j|\theta_j|]$.
Assume without loss of generality that $a_1=0$, then we have by \eqref{elem}, for any fixed $j\in \{2,\ldots,m\}$:
$$
|\muhat(\xi)| \le \prod_{n=1}^{N(\xi)} \Bigl( 1 - 2\pi\eps\Big\|\sum_{k=1}^d \eta_k a_j^{(k)} \theta_k^n \Bigr\|^2\Bigr),
$$
where  $\|\cdot\|$  denotes the distance to the nearest integer. Further, we can assume that all the coordinates of $a_j$ are non-zero; otherwise, we can work in the subspace $$\Hk:=\{x\in \R^d:\
x_k = 0 \iff a_j^{(k)}=0\}$$ and with the
corresponding variables $\theta_k$, and then get the exceptional set of zero $\Lk^d$ measure as a product of a set of zero measure in  
$\Hk$ and the entire $\Hk^\perp$.
 Finally, apply a linear change of variables, so that $a_j^{(k)}=1$ for all $k$, to obtain:
\beq \label{eq-decay3}
|\muhat(\xi)| \le \prod_{n=1}^{N(\xi)} \Bigl( 1 - 2\pi\eps\Big\|\sum_{k=1}^d \eta_k \theta_k^n \Bigr\|^2\Bigr).
\eeq
This is exactly the situation corresponding to the measure $\mu(\thb,\pb)$, and we will be showing (typical) power decay for the right-hand side of \eqref{eq-decay3}.
This completes the reduction.

\smallskip

Next we use a variant of the Erd\H{o}s-Kahane argument, see e.g.\ \cite{PSS00,ShSol16} for other versions of it.
Intuitively, we will get power decay if $\|\sum_{k=1}^d \eta_k \theta_k^n\|$ is uniformly bounded away from zero for a set of $n$'s of positive lower density, uniformly in $\eta$.

Fix $c_1>0$ and $1 < b_1 < b_2 < \infty$, and
consider the compact set 
$$
H = \bigl\{\thb = (\theta_1,\ldots,\theta_d)\in ([-b_2, -b_1]\cup [b_1,b_2])^d:\ |\theta_i-\theta_j|\ge c_1,\ i\ne j \bigr\}.
$$
We will use the notation $[N] = \{1,\ldots,N\},\ [n,N] = \{n,\ldots,N\}$. 
For $\rho,\delta>0$ we define the ``bad set'' at scale $N$:
\beq \label{excep1} E_{H,N}(\delta,\rho)=
\left\{\thb\in H:\ \max_{\eta:\ 1\le \|\eta\|_\infty \le b_2} \frac{1}{N}  
\Bigl|\Bigl\{n\in [N]:\ \Big\|\sum_{k=1}^d \eta_k \theta_k^n \Bigr\| < \rho\Bigr\}\Bigr| >1-\delta\right\}.
\eeq
Now we can define the exceptional set:
$$
\Ek_{H}(\delta,\rho):= \bigcap_{N_0=1}^\infty \bigcup_{N=N_0}^\infty E_{H,N}(\delta,\rho).
$$
Theorem~\ref{th-decay2} will immediately follow from the next two propositions.

\begin{prop} \label{prop-excep}
For any positive $\rho$ and $\delta$, we have $\mu(\thb,\pb) \in \Dsc_d(\alpha)$ 
whenever $\thb\in H\setminus \Ek_H(\delta,\rho)$, where $\alpha$ depends only on $\delta,\rho,H$, and $\eps = \min\{p,1-p\}$.
\end{prop}

\begin{prop} \label{prop-measure}
There exist  $\rho=\rho_H>0$ and $\delta=\delta_H>0$ such that $\Lk^d(\Ek_H(\delta,\rho))=0$.
\end{prop}

\begin{proof}[Proof of Proposition~\ref{prop-excep}]
Suppose that $\thb\in H\setminus \Ek_H(\delta,\rho)$. This implies that there is $N_0\in \N$ such that 
$\thb\not\in E_{H,N}(\delta,\rho)$
for all $N\ge N_0$.
 Let $\xi\in \R^d$ be such that $\|\xi\|_\infty > b_2^{N_0}$. Then $N= N(\xi)\ge N_0$, where $\eta = \eta(\xi) = A^{N(\xi)}\xi$ and $N(\xi)$ is maximal with $\|\eta\|_\infty\ge 1$. 
  From the fact that $\thb\not\in E_{H,N}(\delta,\rho)$ it follows that 
$$
\frac{1}{N}  
\Bigl|\Bigl\{n\in [N]:\ \Big\|\sum_{k=1}^d \eta_k \theta_k^n \Bigr\| < \rho\Bigr\}\Bigr| \le 1-\delta.
$$
Then by \eqref{eq-decay3},
$$
|\muhat(\thb,\pb)(\xi)| \le (1 - 2\pi\eps \rho^2)^{\lfloor\delta N\rfloor}.
$$
By the definition of $N=N(\xi)$ we have 
$$
\|\xi\|_\infty \le b_2^{N+1}.
$$
It follows that
$$
|\muhat(\thb,\pb)(\xi)| = 
O_{H,\eps}(1)\cdot {\|\xi\|}_{_{\scriptstyle{\infty}}}^{-\alpha},
$$
for $\alpha = -\delta\log (1 - 2\pi\eps \rho^2)/\log b_2$, and the proof is complete.
\end{proof}

\begin{proof}[Proof of Proposition \ref{prop-measure}]
It is convenient to express the exceptional set as a union, according to a dominant coordinate of $\eta$ (which may  be non-unique, of course):
$E_{H,N}(\delta,\rho) = \bigcup_{j=1}^d E_{H,N,j}(\delta,\rho)$, where
\beq \label{excep11}
E_{H,N,j}(\delta,\rho) := \left\{\thb\in H:\ \exists\, \eta,\ 1 \le |\eta_j| =  \|\eta\|_\infty\le b_2,\ \frac{1}{N}  
\Bigl|\Bigl\{n\in [N]:\ \Big\|\sum_{k=1}^d \eta_k \theta_k^n \Bigr\| <\rho \Bigr\}\Bigr| >1-\delta\right\}.
\eeq
It is easy to see that $E_{H,N,j}(\delta,\rho)$ is  measurable. Observe that 
$$
\Ek_H(\delta,\rho):= \bigcup_{j=1}^d \Ek_{H,j}(\delta,\rho),\ \ \ \mbox{where}\ \ \ 
\Ek_{H,j}(\delta,\rho):= \bigcap_{N_0=1}^\infty \bigcup_{N=N_0}^\infty E_{H,N,j}(\delta,\rho).
$$
It is, of course, sufficient to show that $\Lk^d(\Ek_{H,j}(\delta,\rho))=0$ for every $j\in [d]$, for some $\delta,\rho>0$. Without loss of generality, assume that $j=d$.
Since $\Ek_{H,d}(\delta,\rho)$ is measurable, the desired claim will follow if we prove that every slice of $\Ek_{H,d}(\delta,\rho)$ in the direction of the $x_d$-axis has zero $\Lk^1$ measure.
Namely, for fixed $\thb' = (\theta_1,\ldots,\theta_{d-1})$  let
$$
\Ek_{H,d}(\delta,\rho,\thb'):= \{\theta_d:\ (\thb',\theta_d)\in \Ek_{H,d}(\delta,\rho)\}.
$$
We want to show that $\Lk^1(\Ek_{H,d}(\delta,\rho,\thb'))=0$ for all $\thb'$.
Clearly,
$$
\Ek_{H,d}(\delta,\rho,\thb'):= \bigcap_{N_0=1}^\infty \bigcup_{N=N_0}^\infty E_{H,N,d}(\delta,\rho,\thb'),
$$
where
\beq \label{excep1} E_{H,N,d}(\delta,\rho,\thb')=
\left\{\theta_d:\ (\thb',\theta_d)\in H:\ \max_{\substack{\\[1.1ex] \eta:\ 1\le |\eta_d|\le b_2\\[1.1ex] \|\eta\|_\infty =|\eta_d|}} \frac{1}{N}  
\Bigl|\Bigl\{n\in [N]:\ \Big\|\sum_{k=1}^d \eta_k \theta_k^n \Bigr\| <\rho \Bigr\}\Bigr| >1-\delta\right\}
\eeq

\begin{lemma} \label{lem-EK}
There exists a constant $\rho>0$ such that, for any $N\in \N$ and $\delta\in (0,\half)$, the set $E_{H,N,d}(\delta,\rho,\thb')$ can be covered
by $\exp(O_H(\delta \log(1/\delta)N))$ intervals of length $b_1^{-N}$.
\end{lemma}

We first complete the proof of the proposition, assuming the lemma.
By Lemma~\ref{lem-EK},
$$
\Lk^1\left(\bigcup_{N=N_0}^\infty E_{H,N,d}(\delta,\rho,\thb')\right) \le \sum_{N=N_0}^\infty \exp(O_H(\delta \log(1/\delta)N))\cdot b_1^{-N} \to 0, \ \ N_0\to \infty,
$$
provided $\delta>0$ is so small that $\log b_1 >O_H(\delta \log(1/\delta))$. Thus $\Lk^1(\Ek_{H,d}(\delta,\rho,\thb'))=0$. \end{proof}

\begin{proof}[Proof of Lemma \ref{lem-EK}]
Fix $\thb'$ in the projection of $H$ to the first $(d-1)$ coordinates and $\eta\in \R^d$, with $1 \le |\eta_d| = \|\eta\|_\infty \le b_2$. 
Below all the constants implicit in the $O(\cdot)$ notation are allowed to depend on $H$ and $d$.
Let $\theta_d$ be such that $(\thb',\theta_d)\in H$ and write
$$
\sum_{k=1}^d \eta_k \theta_k^n = K_n + \eps_n, \ \ n\ge 0,
$$
where $K_n\in \Z$ is the nearest integer to the expression in the left-hand side, so that $|\eps_n| \le \half$. We emphasize that $K_n$ depends on $\eta$ and on $\theta_d$.
Define $A_n^{(0)} = K_n$, $\wt A_n^{(0)} = K_n + \eps_n$, and then for all $n$ inductively: 
\beq\label{induc1}
A_n^{(j)} = A_{n+1}^{(j-1)} - \theta_j  A_{n}^{(j-1)};\ \ \ \ \ \wt A_n^{(j)} = \wt A_{n+1}^{(j-1)} - \theta_j  \wt A_{n}^{(j-1)},\ \ j=1,\ldots, d-1.
\eeq
It is easy to check by induction that
$$
\wt A_n^{(j)} = \sum_{i=j+1}^d \eta_i \prod_{k=1}^j (\theta_i - \theta_k) \theta_i^n,\ \ j=1,\ldots, d-1,
$$
hence
\beq\label{for1}
\wt A_n^{(d-1)} = \eta_d \prod_{k=1}^{d-1} (\theta_d - \theta_k) \theta_d^n;\ \ \ \ \ \theta_d = \frac{\wt A_{n+1}^{(d-1)}}{\wt A_n^{(d-1)}}\,,\ \ \ n\in \N.
\eeq
We have $\|\eta\|_\infty \le b_2$ and $|\wt A_n^{(0)} - A_n^{(0)}| \le |\eps_n|$, and then by induction, by \eqref{induc1},
\beq\label{induc2}
|\wt A_n^{(j)} -A_n^{(j)}| \le (1+b_2)^j \max\{|\eps_n|,\ldots,|\eps_{n+j}|\},  \ j=1,\ldots, d-1.
\eeq
Another easy calculation gives
\begin{eqnarray} \nonumber
K_{n+d+1} & = &  \theta_1 K_{n+d} + A_{n+d}^{(1)}=\cdots \\[1.1ex]
& = & \bigl[\theta_1 K_{n+d} +\theta_2 A_{n+d-1}^{(1)}+ \cdots + \theta_{d-1} A_{n+2}^{(d-2)}\bigr] + A_{n+2}^{(d-1)} 
 \label{for2}
\end{eqnarray}
Since $\frac{ A_{n+2}^{(d-1)}}{A_{n+1}^{(d-1)}}\approx \frac{ \wt A_{n+1}^{(d-1)}}{\wt A_n^{(d-1)}}=\theta_d$, we have
\begin{eqnarray} \label{rational}
K_{n+d+1} & \approx & \bigl[\theta_1 K_{n+d} +\theta_2 A_{n+d-1}^{(1)}+ \cdots + \theta_{d-1} A_{n+2}^{(d-2)}\bigr] +\frac{ {(A_{n+1}^{(d-1)})}^2}{A_n^{(d-1)}} \\
& =: & R_{\theta_1,\ldots,\theta_{d-1}}(K_n,\ldots,K_{n+d}),\nonumber
\end{eqnarray}
\begin{sloppypar}
\noindent where $R_{\theta_1,\ldots,\theta_{d-1}}(K_n,\ldots,K_{n+d})$ is a rational function, depending on the (fixed) parameters $\theta_1,\ldots,\theta_{d-1}$.
To make the approximate equality precise, note that by \eqref{for1} and our assumptions,
$$
\bigl| \wt A_n^{(d-1)}\bigr| \ge c_1^{d-1} b_1^n,
$$
where $b_1 > 1$, and $|\wt A_n^{(d-1)} -A_n^{(d-1)}| \le (1+ b_2)^{d-1}/2$ by \eqref{induc2}. Hence
\beq\label{for3}
\bigl|A_n^{(d-1)}\bigr| \ge c_1^{d-1} b_1^n/2 \ \ \mbox{for}\ n\ge n_0 = n_0(H),
\eeq
and so
$$
\bigl|A_{n+1}^{(d-1)}/A_n^{(d-1)}\bigr|  \le O(1),\ \ n\ge n_0.
$$
\end{sloppypar}
\noindent In the next estimates we assume that $n\ge n_0 (H)$. In view of the above, especially \eqref{induc2} for $j=d-1$,
\begin{eqnarray} \nonumber
\left| \frac{A_{n+1}^{(d-1)}}{A_n^{(d-1)}} - \theta_d\right| & = &  \left| \frac{A_{n+1}^{(d-1)}}{A_n^{(d-1)}} - \frac{\wt A_{n+1}^{(d-1)}}{\wt A_n^{(d-1)}}\right| \\[1.1ex]
& \le & \left| \frac{A_{n+1}^{(d-1)}-\wt A_{n+1}^{(d-1)}}{A_n^{(d-1)}} \right|  + \left|\wt A_{n+1}^{(d-1)}\right|\cdot \left| \frac{1}{A_n^{(d-1)}} - \frac{1}{\wt A_n^{(d-1)}}\right|
\nonumber \\[1.1ex]
& \le & O(1)\cdot\max\{|\eps_n|,\ldots,|\eps_{n+d}|\}\cdot \bigl|A_n^{(d-1)}\bigr|^{-1}. \nonumber
\end{eqnarray}
It follows that, on the one hand, 
\beq\label{esta1}
\left| \frac{A_{n+1}^{(d-1)}}{A_n^{(d-1)}} - \theta_d\right|  \le O(1)\cdot b_1^{-n};
\eeq
and on the other hand, 
\beq\label{esta2}
\left| \frac{{(A_{n+1}^{(d-1)})}^2}{A_n^{(d-1)}} - A_{n+2}^{(d-1)}\right| \le  O(1)\cdot \max\{|\eps_n|,\ldots,|\eps_{n+d+1}|\}.
\eeq
Note that $A_n^{(j)}$, for $j\in [d-1]$, is a linear combination of $K_n, K_{n+1},\ldots,K_{n+j}$ with coefficients that are polynomials in the (fixed) parameters $\theta_1,\ldots,\theta_{d-1}$,
hence the inequality \eqref{esta1} shows that
\smallskip
\beq\label{kuka}
\mbox{\fbox{given $K_n,\ldots,K_{n+d}$,\ \ we have an \ \ $O(1)\cdot b_1^{-n}$-approximation of $\theta_d$.}}
\eeq

\smallskip

\noindent
The inequality \eqref{esta2} yields, using  \eqref{rational} and \eqref{for2}, that, for $n\ge n_0$,
$$
|K_{n+d+1} - R_{\theta_1,\ldots,\theta_{d-1}}(K_n,\ldots,K_{n+d})|\le O(1)\cdot \max\{|\eps_n|,\ldots,|\eps_{n+d+1}|\}.
$$
Thus we have:

\smallskip

\begin{enumerate}
\item[(i)] Given $K_n,\ldots, K_{n+d}$, there are at most $O(1)$ possible values for $K_{n+d+1}$, uniformly in $\eta$ and $\theta_1,\ldots,\theta_{d-1}$. There are also $O(1)$ possible values for
$K_1,\ldots,K_{n_0}$ since $\|\eta\|_\infty$ and $\|\thb\|$ are bounded above by $b_2$.
\item[(ii)] There is a constant $\rho=\rho(H)>0$ such that if $\max\{|\eps_n|,\ldots,|\eps_{n+d+1}|\} < \rho$, then $K_n,\ldots, K_{n+d}$ uniquely determine $K_{n+d+1}$, as the nearest integer to 
$R_{\theta_1,\ldots,\theta_{d-1}}(K_n,\ldots,K_{n+d})$, again independently of $\eta$ and $\theta_1,\ldots,\theta_{d-1}$.
\end{enumerate}

\smallskip

\noindent Fix an $N$ sufficiently large. We claim that for each fixed set $J\subset [N]$ with $|J| \ge (1-\delta)N$, the set
$$
\Bigl\{(K_n)_{n\in [N]}:\ \eps_n=\Big\|\sum_{k=1}^d \eta_k \theta_k^n \Bigr\| < \rho\ \mbox{for some}\ \theta_d, \eta\ \mbox{and all}\ n\in J\Bigr\}
$$
has cardinality $\exp(O(\delta N))$. Indeed, fix such a $J$ and let 
$$
\wt J = \bigl\{i\in [n_0 + (d+1), N]:\ i, i-1, \ldots, i - (d+1)\in J\bigr\}.
$$
We have $|\wt J| \ge (1 - (d+2)\delta)N - n_0 - (d+1)$. If we set 
$$
\Lam_j = (K_i)_{i\in [j]},
$$
then (i), (ii) above show that $|\Lam_{j+1}| = |\Lam_j|$ if $j\in \wt J$ and $|\Lam_{j+1}| = O(|\Lam_j|)$ otherwise. Thus $|\Lam_N| \le O(1)^{(d+2)\delta N}$, as claimed.

The number of subsets $A$ of $[N]$ of size $\ge (1-\delta)N$ is bounded by $\exp(O(\delta\log(1/\delta)N))$ (using e.g.\ Stirling's formula), so we conclude that there are 
$$
\exp(O(\delta\log(1/\delta)N))\cdot \exp(O(\delta N)) = \exp(O(\delta\log(1/\delta)N))
$$
sequences $K_1,\ldots,K_N$ such that $|\eps_n| < \rho$ for at least $(1-\delta)N$ values of $n\in [N]$. Hence by \eqref{kuka} the set \eqref{excep1} can be covered by 
$\exp(O_H(\delta \log(1/\delta)N))$ intervals of radius $b_1^{-N}$, as desired.
\end{proof}

The proof of Theorem~\ref{th-decay2} is now complete.

\bigskip

{\bf Acknowledgement.} Thanks to Ariel Rapaport for corrections and helpful comments on a preliminary version.

\bibliographystyle{plain}
\bibliography{nonunif}

\end{document}